\documentclass[10.5 pt]{article} \hoffset=-1cm \topmargin=-1.5cm
\usepackage{latexsym}
\usepackage{amsthm}
\usepackage{amssymb}
\usepackage{amsmath,float}
\usepackage{tikz}

\newtheorem{theorem}{Theorem}[section]

\newtheorem{thm}[theorem]{Theorem}
\newtheorem{lem}[theorem]{Lemma}

\theoremstyle{definition}

\title{\bf Graphs with Integer Matching Polynomial Zeros }
\author{\bf {S. Akbari$^{{\rm a}}$\footnote{Email addresses: s\underline{~}akbari@sharif.edu, peter.csikvari@gmail.com, ghafaribaghestani\underline{~}a@mehr.sharif.edu, somayeh.ghezelahmad@gmail.com, nahvi\underline{~}mina@mehr.sharif.edu}, P. Csikv\'ari$^{{\rm b }}$\footnote{P. C.  
is partially supported by the Hungarian National Research, Development and Innovation Office, NKFIH grant K109684 and NN114614,  a Slovenian-Hungarian grant,  and by the MTA R\'enyi "Lend\"ulet" Groups and Graphs Research Group, and by the ERC Consolidator Grant  648017.}, A. Ghafari$^{{\rm a}}$, S. Khalashi Ghezelahmad$^{{\rm c }}$, M. Nahvi$^{{\rm a}}$}\\\\
${\rm ^{a}}$\small Department of Mathematical Sciences, Sharif University of Technology, Tehran, Iran\\
${\rm ^{b}}$\small MTA-ELTE Geometric and Algebraic Combinatorics Research Group \& Alfr\'ed R\'enyi Institute of Mathematics 
 \\
${\rm ^{c}}$\small  School of Mathematics, Iran
 University of Science and Technology, Tehran, Iran \\
}

\date{}
\begin{document}
\maketitle

\begin{abstract}
In this paper, we study graphs whose matching polynomial have only integer zeros. A graph is matching integral if the zeros of its matching
polynomial are all integers.  We characterize all matching integral traceable graphs. We show that apart from
$K_7\setminus \left(E(C_3)\cup E(C_4)\right)$ there is no connected $k$-regular matching integral graph if $k\geq 2$.
It is also shown that if $G$ is a graph with a perfect matching, then  its matching polynomial has a zero in the interval $(0,1]$. Finally, we describe all claw-free matching integral graphs.
\end{abstract}


\section{Introduction}


All  graphs we consider are finite, simple and undirected.
Let $G$ be a graph. We denote the edge set and the vertex set of $G$ by $E(G)$ and $V(G)$, respectively.
 By \textit{order} and \textit{size}  of $G$, we mean the number of vertices  and  the number of edges of $G$, respectively.
The maximum  degree of $G$ is denoted by
$\Delta(G)$ (or by $\Delta$ if $G$ is clear from the context).  The minimum degree of $G$ is  denoted by
$\delta(G)$.
In this paper, we denote the complete graph, the path and the cycle of order $n$, by $K_n$, $P_n$ and $C_n$, respectively. The set of neighbors of a vertex $v$ is denoted by $N(v)$.
 A  \textit{traceable graph}, is a graph with a Hamilton path.
An \textit{$r$-matching}  in a graph $G$
is a set of $r$ pairwise non-incident edges. The number of $r$-matchings in $G$ is denoted by $p(G, r)$. The  \textit{matching polynomial} of $G$ is defined by $$\mu(G, x)=\sum_{r=0}^{\lfloor{\frac{n}{2}}\rfloor}(-1)^{r}p(G, r)x^{n-2r},$$ where $n$ is the order of $G$ and $p(G, 0)$ is considered to be $1$, see \cite{[11], [7], [8], [9], [13]}. For instance the matching polynomial of the following graph

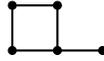
\begin{figure}[H]
\begin{center}
\begin{tikzpicture}[scale=0.3]
\filldraw(-1,0) circle(5pt);
\filldraw(1, 0) circle(5pt);
\filldraw(-1,2) circle(5pt);
\filldraw(1,2) circle(5pt);
\filldraw(3, 0) circle(5pt);

\draw[thick] (-1,0)--(1,0);
\draw[thick] (1, 0)--(1, 2);
\draw[thick] (1, 2)--(-1, 2);
\draw[thick] (-1, 0)--(-1, 2);
\draw[thick] (1, 0)--(3, 0);

\end{tikzpicture}
\end{center}
\end{figure}

\noindent is $\mu(G, x)=x^{5}-5x^{3}+4x$.
By the definition of $\mu(G, x)$, we conclude  that every graph of odd order has $0$ as a matching root.
Furthermore, if $\theta$ is a matching zero of a graph, then so is  $-\theta$. We call a graph,  \textit{matching integral} if all  zeros of its matching polynomial are integers.
A graph is said to be \textit{integral}  if eigenvalues of its adjacency matrix consist entirely of integers. Since 1974, integral graphs have been extensively studied by several authors,  for instance see \cite{[10], [20]}. It is worth mentioning that if $T$ is a tree, then its characteristic polynomial and its matching polynomial are the same, see  \cite[Corollary 1.4, p.21]{[1]}.
Integral trees (so matching integral trees) have been investigated in \cite{[14]}.
\\\\
In Section $2$,
we characterize all traceable graphs which are matching integral.
 In Section $3$, we study  matching integral regular graphs and show that  for $k\geq 2$ there is only one connected  matching integral  $k$-regular graph, namely $K_7\setminus \left(E(C_3)\cup E(C_4)\right)$.
 In order to establish our results, first we need the following theorems:\\
\vspace*{2mm}

\noindent\textbf{Theorem A.} \cite{[5]}\textit{ For any graph $G$, the zeros of $\mu(G, x)$ are all real. If  $\Delta>1$, then the zeros lie in the interval $(-2\sqrt{\Delta-1}, ~ 2\sqrt{\Delta-1})$}.\\

\noindent\textbf{Remark 1.} Let $G$ be a graph. Theorem A implies  that if $\sqrt{\Delta-1}$ is not an integer, then $\mu(G, x)$ contains
at most $2\lfloor 2\sqrt{\Delta-1}\rfloor+1$ distinct integer zeros and if $\sqrt{\Delta-1}$ is an integer, then $\mu(G, x)$ has at most $4\sqrt{\Delta-1}-1$
distinct integer zeros.\\

\noindent\textbf{Theorem B.} \cite[Corollary 1.3, p.97]{[1]} \textit{If $G$ is a connected graph, then the largest zero of  $\mu(G, x)$  has multiplicity $1$. In other words, it is a simple zero.}\\

\noindent Let $t(G)$ be the number of vertices of a longest path in the graph $G$. \\

\noindent\textbf{Theorem C.} \cite[Theorem 4.5, p.107]{[1]}  \textit{ (a) The maximum multiplicity of a zero of $\mu(G, x)$ is at most equal to the number of vertex-disjoint paths required to cover $G$. \\ 
\noindent (b) The number of distinct zeros of $\mu(G, x)$ is at least $t(G)$. \\ 
\noindent (c) In particular, if the graph $G$ is traceable then all zeros of $\mu(G,x)$ is simple.}\\

\noindent\textbf{Theorem D.} \cite{[1]} \textit{If $\theta$ is a zero of $\mu(G,x)$ with multiplicity at least $2$ then for any path $P$ we have that $\theta$ is a zero of $\mu(G\setminus P,x)$, where $G\setminus P$ is the induced subgraph of $G$ on the vertex set $V(G)\setminus V(P)$.} \\

Theorem D is not stated as a theorem in \cite{[1]}, but is used in the proof of Theorem 4.5 of Chapter 6 of \cite{[1]}. Both Theorems C and D rely on the curious identity
$$\mu'(G,x)^2-\mu(G,x)\mu''(G,x)=\sum \mu(G\setminus P,x)^2,$$
where the sum is taken over all paths of G. For instance, if $\theta$ is a zero of $\mu(G,x)$ with multiplicity at least $2$ then it is a zero of both $\mu(G,x)$ and $\mu'(G,x)$ so the left hand side is $0$ at $\theta$, but the right hand side is only $0$ if all terms are $0$.

\section{Matching Integral Traceable Graphs}
In this section, we show that there are  finitely many  matching integral  traceable graphs and  characterize all of them. In fact, we will characterize those graphs whose matching polynomial has only simple integer zeros.
By Theorem C we know that the matching polynomial of a traceable graph has only simple zeros. Hence this way we characterize matching integral traceable graphs.

\begin{thm}
\label{2.1}
Let $G$ be a  connected graph whose matching polynomial has only simple integer zeros. Then $G$ is one of the following graphs: $K_1$, $K_2$, $K_7\setminus \left(E(C_3)\cup E(C_4)\right)$, $G_1$ or $G_2$, where
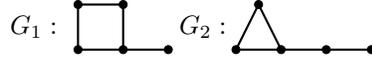
\begin{figure}[H]
\begin{center}
\begin{tikzpicture}[scale=0.3]
\filldraw(-1,0) circle(5pt);
\filldraw(1, 0) circle(5pt);
\filldraw(-1,2) circle(5pt);
\filldraw(1,2) circle(5pt);
\filldraw(3, 0) circle(5pt);
\coordinate [label=left:$G_{1}:$] (O) at (-1.5, 1);
\draw[thick] (-1,0)--(1,0);
\draw[thick] (1, 0)--(1, 2);
\draw[thick] (1, 2)--(-1, 2);
\draw[thick] (-1, 0)--(-1, 2);
\draw[thick] (1, 0)--(3, 0);
\coordinate [label=left:$G_{2}:$] (O) at (6, 1);

\filldraw(6,0) circle(5pt);
\filldraw(8, 0) circle(5pt);
\filldraw(7,2) circle(5pt);
\filldraw(12,0) circle(5pt);
\filldraw(10, 0) circle(5pt);

\draw[thick] (6,0)--(8,0);
\draw[thick] (6, 0)--(7, 2);
\draw[thick] (8, 0)--(7, 2);
\draw[thick] (8, 0)--(10, 0);
\draw[thick] (12, 0)--(10, 0);
\end{tikzpicture}
\end{center}
\label{figurechap1}
\end{figure}
In particular, this is the list of matching integral traceable graphs.

\end{thm}
\begin{proof}
Let $n$ and $m$ be the order and the size of $G$, respectively.
It is enough to prove the first part of the theorem as the second part of the theorem indeed follows from the first one: since $G$ is traceable, by Theorem C, the zeros of  $\mu(G, x)$ are all distinct.
Now, in order to prove the first part, we consider two cases:\\\\
\textbf{Case 1.} $n=2k$, $k\geq 1$.
Since $G$ has even order and all zeros are simple, every zero of $\mu(G, x)$ is different from $0$. Let $\theta_1,\ldots,\theta_k$ be the positive zeros of $\mu(G, x)$. Hence
 \[\mu(G, x)=\prod_{i=1}^{k}(x^2-\theta_i^2)=
x^{2k}-(\theta_1^2+\cdots+\theta_k^2)x^{2k-2}+\cdots+(-1)^{k}\theta_1^2\cdots\theta_k^2.\]
 We have
$$m=\sum_{i=1}^{k}\theta_i^{2}\geq\sum_{i=1}^{k}i^{2}=\frac{k(k+1)(2k+1)}{6}=\frac{n}{12}\left(\frac{n}{2}+1\right)\left(n+1\right).$$
Thus for $n\geq 8$, $m>{ n \choose 2}$,  a contradiction. Now, assume that $n\leq 6$.
We consider three cases:\\\\
\hspace*{4mm}\textbf{Case 1.1.}  $n=2$.
Hence $G= K_2$ and $\mu(G, x)$ has zeros $-1, +1$.\\\\
\hspace*{4mm}\textbf{Case 1.2.}  $n=4$.
Since $\Delta\leq 3$,  ${\rm Theorem \, A}$ implies that  the zeros of  $\mu(G, x)$ lie in $[-2, 2]$.
Hence $\mu(G, x)=(x^{2}-1)(x^{2}-4)=x^4-5x^2+4$. Thus $m=5$ and so $G= K_4\setminus e$, for some edge $e$. But $\mu(K_4\setminus e, x)=x^4-5x^2+2$, a contradiction.\\\\
\hspace*{4mm}\textbf{Case 1.3.}  $n=6$.
Since $\Delta\leq 5$, by  ${\rm Theorem \, A}$ the zeros of  $\mu(G, x)$
 lie in $[-3, 3]$.
Therefore $\mu(G, x)=(x^{2}-1)(x^{2}-4)(x^{2}-9)$. Hence $m=14$ and $G=K_{6}\setminus e$, for some edge $e$. Now, by \cite[Theorem 1.1(b), p.2]{[1]} we know that for any graph $H$ and  its edge $e=(u,v)$ we have
$$\mu(H,x)=\mu(H\setminus e,x)-\mu(H\setminus \{u,v\},x).$$
Applying it to $H=K_6$ and $H\setminus e=G$, $H\setminus \{u,v\}=K_4$  we find that
$$\mu(G, x)=\mu(K_6, x)+\mu(K_4, x)=(x^6-15x^4+45x^2-15)+(x^4-6x^2+3),$$
 a contradiction.\\\\
\textbf{Case 2.}  $n=2k+1$, $k\geq 0$.
Let $ 0, \theta_1,\ldots,\theta_k$ be the non-negative zeros of
$\mu(G, x)$. Thus $$\mu(G, x)=x\prod_{i=1}^{k}(x^2-\theta_i^{2})=x^{2k+1}-(\theta_1^2+\cdots+\theta_k^2)x^{2k-1}+\cdots+(-1)^{k}\theta_1^2\cdots\theta
 _k^2\,x.$$ It follows that
$$m=\sum_{i=1}^{k}\theta_i^{2}\geq\sum_{i=1}^{k}i^{2}=\frac{n(n^{2}-1)}{24}.$$
Hence for $n\geq 13$, $m>{ n \choose 2}$,  a contradiction.
If $n\leq 3$, then clearly  $G=K_1$.  Now, we consider four cases: \\\\
\hspace*{4mm}\textbf{Case 2.1.}  $n=5$.
Since $\Delta\leq 4$, by ${\rm Theorem \, A}$ the zeros of $\mu(G, x)$ lie in $[-3, 3]$. Since $m=\theta_1^2+\theta_2^2\leq 10$, the positive zeros of $\mu(G, x)$ are either $1, 3$ or $1, 2$. In the first case, $\mu(G, x)=x(x^2-1)(x^2-9)$. Hence $m=10$ and so  $G= K_5$. Therefore, $p(G, 2)=\frac{5!}{2!2^2}\not=9$, a contradiction. In the second case, $\mu(G, x)=x(x^2-1)(x^2-4)$. Hence $m=5$ and $p(G, 2)=4$.
 It follows that $G$ is one of the following graphs:
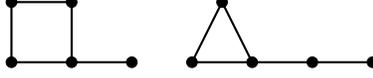
\begin{figure}[H]
\begin{center}
\begin{tikzpicture}[scale=0.4]
\filldraw(0,0) circle(5pt);
\filldraw(2, 0) circle(5pt);
\filldraw(0,2) circle(5pt);
\filldraw(2,2) circle(5pt);
\filldraw(4, 0) circle(5pt);

\draw[thick] (0,0)--(2,0);
\draw[thick] (2, 0)--(2, 2);
\draw[thick] (2, 2)--(0, 2);
\draw[thick] (0, 0)--(0, 2);
\draw[thick] (2, 0)--(4, 0);

\filldraw(6,0) circle(5pt);
\filldraw(8, 0) circle(5pt);
\filldraw(7,2) circle(5pt);
\filldraw(12,0) circle(5pt);
\filldraw(10, 0) circle(5pt);

\draw[thick] (6,0)--(8,0);
\draw[thick] (6, 0)--(7, 2);
\draw[thick] (8, 0)--(7, 2);
\draw[thick] (8, 0)--(10, 0);
\draw[thick] (12, 0)--(10, 0);
\end{tikzpicture}
\label{figurechap1}
\end{center}
\end{figure}

\textbf{Case 2.2.}  $n=7$. Since
$\Delta\leq 6$,  ${\rm Theorem \, A}$ implies that  the zeros of $\mu(G, x)$  lie in $[-4, 4]$. Since $m=\theta_{1}^{2}+\theta_{2}^{2}+\theta_{3}^{2}\leq 21$, the positive zeros of $\mu(G, x)$ are either $1, 2, 3$ or
$1, 2, 4$. In the first case,  $\mu(G, x)=x(x^{2}-1)(x^{2}-4)(x^{2}-9)$.
So  $m=14$, $p(G, 2)=49$  and  $p(G, 3)=36$. On the other hand, $ |E(K_7\setminus (E(C_3)\cup E(C_4)))|=14 $,  $ p(K_7\setminus (E(C_3)\cup E(C_4)),2)=49 $ and $p(K_7\setminus (E(C_3)\cup E(C_4)),3)=36 $. Since $K_7\setminus (E(C_3)\cup E(C_4))$ is $ 4 $-regular,
\cite[Exercise 4, p.15]{[1]} implies that $G$ is $4$-regular. Obviously, there are two non-isomorphic $4$-regular graphs of order $7$,
$K_7\setminus E(C_7)$ and $K_7\setminus (E(C_3)\cup E(C_4))$. Therefore
 $G=K_7\setminus (E(C_3)\cup E(C_4))$.\\

In the second case, $\mu(G, x)=x(x^{2}-1)(x^{2}-4)(x^{2}-16)$ and $m=21$. Hence
 $G= K_7$ and so $p(G,3)=\frac{7!}{3!2^{3}}\not=64$, a contradiction.  \\\\
\hspace*{4mm}\textbf{Case 2.3.}  $n=9$.
Since $\Delta\leq 8$,  by ${\rm Theorem \, A}$  the zeros of $\mu(G, x)$  lie in $[-5, 5]$. Note that $m=\theta_{1}^{2}+\theta_{2}^{2}+\theta_{3}^{2}+\theta_4^2\leq 36$, so the positive zeros of $\mu(G, x)$ should be
$1, 2,3,4$. Thus $\mu(G, x)=x(x^2-1)(x^2-4)(x^2-9)(x^2-16)$. Hence $m=30$ and $p(G, 2)=273$. Furthermore,
$p(G, 2)=\begin{pmatrix}
30 \\ 2
\end{pmatrix}-\sum_{i=1}^{9}
\begin{pmatrix}
d_i \\ 2
\end{pmatrix}$, where $d_1,\ldots, d_9$ is the degree sequence of $G$. So we have the following:

\[\sum_{i=1}^{9}d_i=60,\ \ \ \mbox{and}\ \ \  \sum_{i=1}^{9}d_i^2=384.\]
But this contradicts the Cauchy--Schwarz inequality:
$$60^2=\left(\sum_{i=1}^{9}1\cdot d_i\right)^2\leq \left(\sum_{i=1}^{9}1^2\right)\left(\sum_{i=1}^{9} d_i^2\right)=9\cdot 384<60^2.$$ \\\\
\hspace*{4mm}\textbf{Case 2.4.}  $n=11$.
Since $\Delta\leq 10$, by ${\rm Theorem \, A}$  the zeros of $\mu(G, x)$  lie in $[-5, 5]$. Thus  $\mu(G, x)=x(x^{2}-1)(x^{2}-4)\cdots(x^{2}-25)$. Hence $m=55$ and so $G=K_{11}$. Therefore
$p(G, 5)=\frac{11!}{5!2^5}\not=14400$, a contradiction.
\end{proof}
\section{Matching Integral Regular Graphs}

In this section, we study matching integral regular graphs. We show the following theorem.

\begin{thm} \label{regular}
If $G$ is a matching integral $k$-regular graph $(k \geq 2)$ then it is disjoint union of $K_7\setminus \left(E(C_3)\cup E(C_4)\right)$.
\end{thm}

Let $G$ be a  graph of order $n$. Recall that $t(G)$ denotes the number of vertices of a longest path in the graph $G$. By Theorem A all zeros of a matching integral $k$--regular graph lie in the interval $(-2\sqrt{k-1},2\sqrt{k-1})$ and so the number of distinct zeros is at most $2\lfloor 2\sqrt{k-1}\rfloor +1$. By the second claim of Theorem C this is an upper bound for $t(G)$, hence
$$t(G)\leq  2\lfloor 2\sqrt{k-1}\rfloor +1.$$
On the other hand, $t(G)\geq k+1$ for a $k$--regular graph simply by choosing the vertices of a path greedily. This already gives that $k\leq 14$. On the other hand one can improve on the bound $t(G)\geq k+1$.
The following lemma is practically an immediate consequence of a theorem of Dirac.

\begin{lem} \label{long-path} Let $G$ be a $k$--regular connected graph on $n$ vertices. Then
$$t(G)\geq \min(2k+1,n).$$
\end{lem}

Before we prove this lemma let us deduce the following corollary.

\begin{lem} Let $G$ be a matching integral $k$--regular graph, where $k\ge 3$. Then $G$ is a disjoint union of a few copies of $K_7\setminus \left(E(C_3)\cup E(C_4)\right)$.
\end{lem}

\begin{proof} We can assume that $G$ is connected since the set of matching zeros of a graph is the union of the set of matching zeros of the components of the graph. If $t(G)=n$, where $n$ is the number of vertices then $G$ is traceable and so by Theorem~\ref{2.1} it is $K_7\setminus \left(E(C_3)\cup E(C_4)\right)$. If $t(G)\neq n$ then $t(G)\geq 2k+1$. Then
$$2k+1\leq t(G)\leq 2\lfloor 2\sqrt{k-1}\rfloor +1$$
implies that $k\le 2$ contradicting the condition $k\geq 3$.
\end{proof}

Next we prove Lemma~\ref{long-path}. We will use the following theorem of Dirac.

\begin{lem}[\cite{D}] \label{long-cycle} Let $c(G)$ be the longest cycle of a connected graph $G$ on $n$ vertices. Assume that $G$ is $2$-connected and has minimum degree at least $k$. Then $c(G)\geq \min(2k,n)$.
\end{lem}

\begin{proof}[Proof of Lemma~\ref{long-path}]
If $G$ is $2$-connected then by Lemma~\ref{long-cycle} we have $c(G)\geq \min(2k,n)$. If $c(G)=n$ then clearly
$t(G)=n$. If $c(G)\geq 2k+1$ we are done again. If $c(G)=2k<n$ then by connectedness of $G$ there is a vertex $v$ not in the cycle which is connected to some vertex of the cycle, but then there is a path of length at least $2k+1$.

If $G$ is not $2$-connected then it contains a cut vertex $v$. From $v$ let us build up two vertex disjoint paths in two different components of $G-v$ greedily. By concatenating the two paths we get a path of length at least $2k+1$.
\end{proof}

\vspace*{1 mm}

Finally for $k=2$ the following lemma is an immediate consequence of Theorem \ref{2.1}.
\begin{lem}
\label{2.9}
For every positive integer $n$, $C_n$ is not matching integral.
\end{lem}

\section{Matching Integral Graphs with a Perfect Matching}

In this section we study the zeros of a matching polynomial of  a graph with a perfect matching.

\begin{thm} \label{perfect} If a graph $G$ has a perfect matching then its matching polynomial has a zero in the interval $(0,1]$. If it has no zero in the interval $(0,1)$ then it is the disjoint union of some $K_2$.
\end{thm}

\begin{proof} Let $G$ be a graph on $2n$ vertices. Since $G$ has a perfect matching we have $p(G,n)\neq 0$, consequently
$$\mu(G,x)=\prod_{i=1}^n(x^2-\theta_i^2),$$
where $\theta_i\neq 0$. Then
$$\frac{p(G,n-1)}{p(G,n)}=\sum_{i=1}^n\frac{1}{\theta_i^2}.$$
Next we show that
$$\frac{p(G,n-1)}{p(G,n)}\geq n.$$
Indeed, every perfect matching contains exactly $n$ matchings of size $n-1$, and every matching of size $n-1$ can be extended to a perfect matching in at most $1$ way. Hence
$$n\leq  \frac{p(G,n-1)}{p(G,n)}=\sum_{i=1}^n\frac{1}{\theta_i^2}$$
implies that $\min_i\theta_i^2\leq 1$, and if $\min_i\theta_i^2=1$ then $\theta_1^2=\dots =\theta_n^2=1$. Hence the graph has $\theta_1^2+\dots +\theta_n^2=n$ edges which form a perfect matching, i. e., $G$ is the disjoint union of some $K_2$.
\end{proof}

We offer one more theorem in the same spirit. Let
$$f(t)=\left\{ \begin{array}{ll} t+1 & \mbox{if}\ t\neq 1, \\ 3 & \mbox{if}\ t=1. \end{array} \right. $$
The proof of the following theorem is practically the same as the proof of Theorem~\ref{perfect}. The only extra observation one needs is that a graph on $t+2$ vertices without a $2$--matching contains at most $f(t)$ edges.

\begin{thm} Let $G$ be a graph with at least one edge. Assume that the multiplicity of $0$ as a zero
 of the matching polynomial of the graph $G$ is $t$. Then the interval $(0,\sqrt{f(t)}]$ contains a zero of the matching polynomial of $G$.
\end{thm}

\section{Matching Integral Claw-free Graphs}

In this section we study matching integral claw-free graphs.

\begin{thm} \label{claw-free} Let $G$ be a connected matching integral claw-free graph. Then $G$ is one of $K_1,K_2$ or $G_2$.
\end{thm}

Note that from the list of traceable matching integral graphs $G_1$ and $K_7\setminus \left(E(C_3)\cup E(C_4)\right)$ are not claw-free. \\

The proof of Theorem~\ref{claw-free} is quite long, and so we summarize here the plan of the proof. First we show that a connected claw-free graph always contains a matching which avoids at most one vertex, so if $G$ has even order then it contains a perfect matching, and if it has odd order then the largest matching avoids exactly one vertex. This settles the multiplicity of $0$ as a zero of the matching polynomial. If $G$ contains a perfect matching then Theorem~\ref{perfect} already gives that $G$ is $K_2$. Then just as in case of regular graphs we try to find long paths in the graph. We fix the largest degree $\Delta$ and as we gain more and more information about the length of the longest path we exclude the possibility of more and more values of $\Delta$. By the time we stuck with the ideas of finding long paths we will have enough information about the structure of $G$ so that we get many information about the matching polynomial. This way we can shrink the set of possible matching polynomials. Then we translate it back to structural information about the graph and we finish the proof. \\

\begin{lem} Let $G$ be a connected claw-free graph. Then it contains a matching which avoids at most one vertex.
\end{lem}

\begin{proof} This is a well-known statement which can be found in the book of J. Akiyama and M. Kano \cite{AK}.
See also \cite{L,S}.
\end{proof}

Next we start our hunting to long  paths. Our main tool is the following lemma and its corollary.

\begin{lem} \label{claw-free-neighbor} Let $H$ be a graph such that $H$ has largest independent set of size at most $2$. Then $H$ has either a Hamiltonian  cycle or there are two vertex-disjoint cliques in $H$ covering all vertices of $H$.
\end{lem}

\noindent To prove Lemma~\ref{claw-free-neighbor} we will use the following theorem of Chv\'atal and Erd\H os \cite{CE}. It can also be found in the book of Bondy and Murty \cite{BM} as Theorem 18.10.\\

\begin{lem}[\cite{CE}] \label{ChEr} Let $G$ be a graph on at least $3$ vertices. If, for some $s$, $G$ is $s$-connected and contains no independent set of more than $s$ vertices, then $G$ has a Hamiltonian cycle. 
\end{lem}

\begin{proof}[Proof of Lemma~\ref{claw-free-neighbor}]
Clearly, if $G$ is $2$-connected then with the choice of $s=2$ Lemma~\ref{ChEr} immediately implies our statement. If $G$ is not connected then it is easy to see that $G$ must be the union of two disjoint cliques. While, if $G$ has a cut vertex $v$, then $G-v$ must be the union of two disjoint cliques and $v$ must be adjacent to all elements of at least one of the cliques. So we are done in this case too.
\end{proof}

\begin{lem} \label{long-path2} (a) Let $G$ be a connected  claw-free graph with a vertex $v$ of largest degree $\Delta$. Then for any $u\in N(v)$ there is a path $P_u$ starting at $u$ which covers all vertices of $N(v)\cup \{v\}$.

\noindent (b) We have $t(G)\geq \Delta+1$, and if $V(G)\neq N(v) \cup \{v\}$ then $t(G)\geq \Delta+2$.
\end{lem}

\begin{proof} Part (b) is an immediate consequence of part (a) since $|N(v)\cup \{v\}|=\Delta+1$, and if $G\neq N(v)\cup \{v\}$ then by the connectedness of $G$ there is a $w$ adjacent to some $u\in N(v)$ so we can extend $P_u$ with $w$.

To prove part (a), let $H$ be the graph induced by the neighbors of $v$. Then $H$ has largest independent set of size at most $2$, otherwise there would be a claw in $G$ with center $v$.  Then by the previous lemma $H$ has either a Hamiltonian cycle or there are two vertex-disjoint cliques in $H$ covering all vertices of $H$. In both cases there is a path $P_u$ starting at vertex $u$ which covers $V(H)\cup \{v\}$.
\end{proof}

After all this praparation we are ready to prove Theorem~\ref{claw-free}.

\begin{proof}[Proof of Theorem~\ref{claw-free}.]
For a positive integer $r$ let $s(r)$ be the number of integers in the interval $(-2\sqrt{r-1},2\sqrt{r-1})$. Clearly, this is $2\lfloor 2\sqrt{r-1}\rfloor +1$ if $r-1$ is not a perfect square, and $4\sqrt{r-1}-1$ if $r-1$ is a perfect square. So far we know that
$$\Delta+1\leq t(G)\leq s(\Delta).$$
From this, it follows that $2\leq \Delta\leq 14$. Furthermore, if $\Delta+1=s(\Delta)$ then $t(G)=\Delta+1$ and then $G= N(v) \cup \{v\}$. In particular, it is traceable, but we already described all traceable claw-free graphs: $K_1,K_2$ and $G_2$ (and their largest degrees are $0,1$ and $3$). The following table shows the values of $s(\Delta)$ and $\Delta+1$. The case $\Delta=13$ contradicts $\Delta+1\leq s(\Delta)$.
This table also shows that we can exclude the possibility of $\Delta
\in \{2,10,12,14\}$, because in these cases $\Delta+1=s(\Delta)$.

\begin{center}
\begin{tabular}{|c|ccccccccccccc|} \hline
$r$ & 2 & 3 & 4& 5 & 6 & 7 & 8 & 9 & 10 & 11 & 12 & 13 & 14  \\ \hline
$2\sqrt{r-1}$  & 2 & 2.82 & 3.46 & 4 & 4.47 & 4.89 & 5.29 & 5.65 & 6& 6.32 & 6.63 & 6.92 & 7.21  \\ \hline
$s(r)$ &  3 & 5 & 7& 7 & 9 & 9 & 11 & 11 & 11& 13 & 13 & 13 & 15  \\ \hline
$s(r)-(r+1)$ &  0 & 1 & 2& 1 & 2 & 1 & 2 & 1 & 0& 1 & 0 & -1 & 0  \\ \hline
\end{tabular}
\end{center}

Our next goal is to show  that it is not possible that $k(\Delta)=1$, and if $k(\Delta)=2$ then the multiplicity of some $\theta_i$ with $|\theta_i|>1$ is at most $1$. Note that $k(\Delta)=1$ if $\Delta\in \{3,5,7,9,11\}$, and $k(\Delta)=2$ if $\Delta \in \{4,6,8\}$.
\medskip

Let $k(r)=s(r)-(r+1)$, it expresses how much longer the longest path can be than $\Delta+1$ if $r=\Delta$.
As before let $v$ be a vertex of degree $\Delta$ and $N(v)$ denotes its set of neighbors.
We will show that if $k(\Delta)=1$ then $G\setminus (N(v)\cup \{v\})$ is an empty graph on some vertices, and if
$k(\Delta)=2$ then $G\setminus (N(v)\cup \{v\})$ is a union of $K_1$ and $K_2$.

For a vertex $u\in N(v)$ let $G_u$ be the graph obtained from $G$ by deleting $N(v)\cup \{v\}$ except $u$. Let $C_u$ be the component of $u$ in $G_u$. We show that if $k(\Delta)=1$ then $|C_u|\leq 2$, and if $k(\Delta)=2$ then $|C_u|\leq 3$. Let $u_1,\dots ,u_r$ be the neighbors of $u$ in $C_u$. The vertices $u_1,\dots ,u_r$ have to form a clique otherwise there would be a claw $u,v,u_i,u_j$ with center $u$ for some non-adjacent $u_i$ and $u_j$. This means that we can add a path $u_1u_2\dots u_r$ to the path $P_u$ covering $N(v)\cup \{v\}$. In particular,  this means that in the case  $k(\Delta)=1$ we have $r\leq 1$, and in the case  $k(\Delta)=2$ we have $r\leq 2$. If
$k(\Delta)=1$ and $r=1$ the vertex $u_1$ cannot have further neighbor in $C_u$ so $|C_u|=2$. Similarly, if $k(\Delta)=2$ and $r=2$ the vertices $u_1$ and $u_2$ cannot have further neighbors in $C_u$: if some $w$ is adjacent to say $u_1$ then $wu_1u_2P_u$ be a path of length $\Delta+4$. If $k(\Delta)=2$ and $r=1$ the vertex $u_1$ can have at most one further neighbor, but not more: if $u_2$ and $u_3$ are adjacent to $u_1$ then they are adjacent to each other too since otherwise we have a claw $u_1,u,u_2,u_3$ with center $u_1$, but then $u_2,u_3,u_1,P_u$ is again a path of length $\Delta+4$. Finally, if $k(\Delta)=2$ and $r=1$, and the vertex $u_1$ has a further neighbor $u_2$, then $u_2$ cannot have further neighbors in $C_u$, so $|C_u|\leq 3$.
This proves that if $k(\Delta)=1$ then $|C_u|\leq 2$, and if $k(\Delta)=2$ then $|C_u|\leq 3$. In particular, it shows that if $k(\Delta)=1$ then $G\setminus (N(v)\cup \{v\})$ is a disjoint union of $K_1$, and if
$k(\Delta)=2$ then $G\setminus (N(v)\cup \{v\})$ is a disjoint union of $K_1$ and $K_2$.
\medskip

Now let $P$ be the path going through $N(v)\cup \{v\}$. If $\theta$ is a zero of $\mu(G,x)$ with multiplicity at least $2$ then by Theorem D we have that $\theta$ is a zero of
$$\mu(G\setminus P,x)=\mu(K_1,x)^{\alpha_1}\mu(K_2,x)^{\alpha_2}=x^{\alpha_1}(x^2-1)^{\alpha_2},$$
where $\alpha_2=0$ if $k(\Delta)=1$.
If $k(\Delta)=1$ then it means that $\theta=0$, and if $k(\Delta)=2$ then it means that $\theta\in \{-1,0,1\}$. Since in a connected claw-free graph there is always a matching which avoids at most $1$ vertex we have that the multiplicity of $0$ as a zero is at most $1$. This means that if $k(\Delta)=1$ then all zeros are simple, but then by Theorem~\ref{2.1} we know the complete list of these graphs. In the list the only claw-free graphs were  $K_1,K_2,G_2$.
\medskip

Next we study the case of $k(\Delta)=2$. Recall that it follows from the table on page 7 that the case $k(\Delta)=2$ implies that $\Delta\in\{4,6,8\}$. Note that in this case all integers in the interval $(-2\sqrt{\Delta-1},2\sqrt{\Delta-1})$ must be a zero of $\mu(G,x)$ since otherwise the number of distinct zeros is at most $s(\Delta)-2\leq \Delta+1$ (the missing zero cannot be $0$ since if $G$ contains a perfect matching, then we are done by Theorem~\ref{perfect}), and so $t(G)=\Delta+1$. In this case Part (b) of Lemma~\ref{long-path2} implies that $V(G)=N(v)\cup \{v\}$, and  Part (a) of this lemma shows that the graph is traceable, and we have discussed this case already.

So we have the remaining cases for $\Delta=4,6,8$:
$$\mu(G,x)=x(x^2-1)^{\alpha}(x^2-4)(x^2-9),$$
$$\mu(G,x)=x(x^2-1)^{\alpha}(x^2-4)(x^2-9)(x^2-16),$$
$$\mu(G,x)=x(x^2-1)^{\alpha}(x^2-4)(x^2-9)(x^2-16)(x^2-25).$$

Next we prove that the cases $\Delta=6,8$ are not possible and in case of $\Delta=4$ there are $7$ vertices of degree $4$, all other vertices are of degree $1$.

Let $a_i$ be the number of vertices with degree $i$. Then $a_i=0$ if $i>\Delta$ or $i=0$ ($G$ is connected), and $\sum a_i=n$ is the number of vertices; $\sum ia_i=2m$, two times the number of edges; and $\sum i^2a_i=\sum d(u)^2=m(m+1)-2p(G,2)$.

In particular, for $\Delta=4$ we have
$$\sum_{i=1}^4a_i=5+2\alpha,\ \ \ \ \sum_{i=1}^4ia_i=26+2\alpha, \ \ \ \sum_{i=1}^4i^2a_i=110+2\alpha.$$
Then
$$0\geq -2(a_2+a_3)=\sum_{i=1}^4(i-1)(i-4)a_i=\left(\sum_{i=1}^4i^2a_i\right)-5\left(\sum_{i=1}^4ia_i\right)+4\left(\sum_{i=1}^4a_i\right)=0,$$
so $a_2=a_3=0$ and $a_4=7$, $a_1=2(\alpha-1)$.

For $\Delta=6$ we have
$$\sum_{i=1}^6a_i=7+2\alpha,\ \ \ \ \sum_{i=1}^6ia_i=58+2\alpha, \ \ \ \sum_{i=1}^6i^2a_i=382+2\alpha.$$
Then
$$0\geq \sum_{i=1}^6(i-1)(i-6)a_i=18,$$
a contradiction.

For $\Delta=8$ we have
$$\sum_{i=1}^8a_i=9+2\alpha,\ \ \ \ \sum_{i=1}^8ia_i=108+2\alpha, \ \ \ \sum_{i=1}^8i^2a_i=1032+2\alpha.$$
Then
$$0\geq \sum_{i=1}^8(i-1)(i-8)a_i=132,$$
again contradiction.
\medskip

Next we eliminate the case $\Delta=4$. In this case we know that $7$ vertices have degree $4$, and the rest of the vertices have degree $1$. Clearly, two degree $1$ vertices cannot be adjacent, because $G$ is connected. One vertex of degree $4$ cannot be adjacent to two vertices of degree $1$ since there would be a claw in the graph. We can assume that there are indeed vertices of degree $1$ since otherwise all zeros would be simple. In this case we know from Theorem~\ref{2.1} that the only candidate is $K_7\setminus E(C_3)\cup E(C_4)$ which is not claw-free. If we delete all degree $1$ vertices we get a graph $G'$ on $7$ vertices where all degrees are $3$ or $4$. Both $3$ and $4$ should exist since we assumed the existence of degree $1$ vertices, and it cannot occur that all degrees are $3$ since the number of edges would be $7\cdot 3/2=10.5$. Clearly, $G'$ is still claw-free as it is an induced subgraph of $G$.

We show that $G'$ contains a Hamiltonian cycle. $G'$ is clearly connected since $G$ is connected. We show that $G'$ is also $2$-connected. Indeed, if there were a cut-vertex $u$ then since the minimum degree of $G'$ is at least $3$, the only possibility is that $G'-u$ has $2$ components of size $3$ and every vertex is connected to $u$, but then $u$ has degree $6$. Then again we use Dirac's theorem, Lemma~\ref{long-cycle}: the graph contains a cycle $C$ of length at least $6$. If the length is $7$ then we are done. If the length of $C$ is $6$, then let the vertices of $C$ be $v_1\dots v_6$ and $v_7$ be the remaining vertex. The degree of $v_7$ is at least $3$, if it is adjacent to $2$ neighboring vertices of $C$ then we can extend $C$ to a Hamiltonian cycle. Since the degree of $v_7$ is at least $3$, the only case we have to consider is when $v_7$ is adjacent to vertices of even indices, or those of odd indices. By symmetry we can assume that $v_7$ is adjacent to $v_1,v_3,v_5$. Next observe that if any $2$ of $v_2,v_4,v_6$ are adjacent then there is a Hamiltonian cycle: if $v_2$ and $v_6$ are adjacent then $v_1v_2v_6v_5v_4v_3v_7v_1$ is a Hamiltonian cycle, the other two cases are symmetric to this one. Since the minimum degree of $G'$ is at least $3$, the vertex $v_2$ has to be adjacent to $v_5$, the vertex $v_4$ has to be adjacent to $v_1$, and $v_6$ has to be adjacent to $v_3$. The obtained graph is actually a $K_{3,4}$ with classes $v_1,v_3,v_5$ and $v_2,v_4,v_6,v_7$. This graph doesn't contain a Hamiltonian cycle, but it contains a lot of claws, for instance $v_1,v_2,v_4,v_6$. If we add one more edge which we can assume to be $v_2v_4$ by symmetry then it will contain a Hamiltonian cycle: $v_1v_7v_3v_6v_5v_4v_2v_1$.

Now we can show a path consisting of $8$ vertices in $G$: start at some pendant vertex, and after jumping to its unique neighbor  go through a Hamiltonian path of $G'$. Hence $t(G)=8>7=s(4)$, so $G$ cannot be matching integral. Hence we eliminated the case $\Delta=4$ too. We are done.
\end{proof}

We conclude  the paper with the following question:\\

\noindent {Question.} Is it true that there are finitely many  matching integral $2$-connected graphs?
\bigskip

\noindent \textbf{Acknowledment.} We thank the referees for their careful reading and helpful comments.
\bigskip


\end{document}